\documentclass[envcountsame,
envcountsect
]{llncs}

\usepackage{amsmath,amssymb}
\usepackage{latexsym}
\usepackage{amsfonts} 
\usepackage{pifont}
\usepackage{hyperref}
\usepackage{rrpo}

\title{Predicative Lexicographic Path Orders
\thanks{This is the full version of the extended abstract 
\cite{Eguchi13} that appeared in the proceedings of the 13th International Workshop on
Termination (WST 2013).
This work is supported by Grant-in-Aid for JSPS Fellows (Grant No.
$25 \cdot 726$) that is granted at Graduate School of Science, Chiba University, Japan.}%
}

\subtitle{An Application of Term Rewriting to the Region of Primitive
Recursive Functions}

\author{Naohi Eguchi%
}
\institute{Institute of Computer Science, University of Innsbruck%
\\
Technikerstrasse 21a, 6020 Innsbruck, Austria 
\\
\email{naohi.eguchi@uibk.ac.at}
}

\begin{document}

\maketitle

\begin{abstract}
In this paper we present a novel termination order the {\em predicative lexicographic path order} (PLPO for short), a syntactic restriction of
the lexicographic path order.
As well as lexicographic path orders, 
several non-trivial primitive recursive equations, e.g., primitive
 recursion with parameter substitution,
unnested multiple recursion,
 or simple nested recursion, can be oriented with PLPOs.
It can be shown that the PLPO however only induces primitive recursive upper bounds on derivation
 lengths of compatible rewrite systems.
This yields an alternative proof of a classical fact that the class of
 primitive recursive functions is closed under those non-trivial
 primitive recursive equations. 
\end{abstract}

\section{Introduction}

As observed by E. A. Cichon and A. Weiermann 
\cite{cich_weier}, in order to assess the {\em time
resources} required to compute a function,
one can discuss lengths of rewriting sequences, which is
known as {\em derivation lengths}, in a term
rewrite system defining the function.
More precisely, if the maximal derivation length of a given rewrite system $\mathcal R$
is bounded by a function in a class $\mathcal F$, then the function
defined by $\mathcal R$ is elementary recursive in $\mathcal F$.
In \cite{AM10}, M. Avanzini and G. Moser have sharpened this connection showing that ``elementary
recursive in''
can be replaced with ``polynomial time in'' if one only considers innermost
rewriting sequences starting with terms whose arguments are already
normalised.
Based on the clear connection between time complexity of functions and
derivation lengths of term rewrite systems,
{\em complexity analysis by term rewriting} has been developed, e.g. 
\cite{AM08,AEM11,AEM12}, also providing machine-independent logical
characterisations of complexity classes related to polynomial-time
computable functions.

In most known cases, those term-rewriting characterisations are more flexible than
purely recursion-theoretic characterisations, and thus non-trivial
closure conditions might be deduced.
In this paper we present an application of rewriting techniques to
some closure conditions for primitive recursive functions.
It is known that the class of primitive recursive
functions is closed under a recursion schema that is not an instance of
primitive recursion, e.g., primitive recursion with parameter
substitution (\textbf{PRP}), unnested multiple
recursion (\textbf{UMR}), or simple nested recursion (\textbf{SNR}):

$\begin{array}{lrcl}
 (\mathbf{PRP}) & 
 f(x+1, y) & = & h(x, y, f(x, p(x, y))) \\
 (\mathbf{UMR}) &
 f(x+1, y+1) & = &
 h(x, y, f(x, p(x, y)), f(x+1, y)) \\
 (\mathbf{SNR}) &
 f(x+1, y) & = &
 h(x, y, f(x, p(x, y, f(x, y))))
 \end{array}
$

\noindent
In the equation of (\textbf{PRP}),
in contrast to a standard equation of primitive recursion
$f(x+1, y) = h(x, y, f(x, y))$,
the second argument of $f(x, y)$ can be parameterised with another
function $p$, in (\textbf{UMR}) recursive calls on multiple arguments
are allowed but nested forms of recursion are forbidden, and in
(\textbf{SNR}) nested recursion is allowed but substitution of recursion
terms for recursion arguments is forbidden in contrast to the
general form of nested recursion. 
Note that any of (\textbf{PRP}), (\textbf{UMR}) and (\textbf{SNR}) is an
instance of (nested) multiple recursion.
The proofs of these facts are traced back to R. P\'eter's work \cite{peter},
where for each of (\textbf{PRP}), (\textbf{UMR}) and (\textbf{SNR}), a
tricky recursion-theoretic reduction to the standard primitive recursion
was achieved.

H. Simmons \cite{Simmons88} provided uniform proofs of P\'eter's
results in a higher order setting.
In \cite{cich_weier}, alternative proofs of P\'eter's results were given
employing primitive recursive number-theoretic interpretations of rewrite systems
corresponding to the non-trivial primitive recursive equations mentioned above.
On the other side, in order to look into the distinction between
primitive recursive and P\'eter's non-primitive recursive multiply
recursive functions,
it is of interest to discuss (variants of) a termination order
known as the lexicographic path order (LPO
for short).
As shown by Weiermann \cite{weier95}, the LPO induces multiply recursive upper
bounds on derivation lengths of compatible rewrite systems.
Note that any equation of (\textbf{PRP}), (\textbf{UMR}) and
(\textbf{SNR}) can be oriented with an LPO.
Hence it is natural to restrict the LPO to capture these non-trivial
primitive recursive equations.

Stemming from Simmons' approach in \cite{Simmons88} but without
higher-order notions. we introduce the
{\em predicative lexicographic path order} (PLPO for short),
a syntactic restriction of the LPO.
As well as LPOs, 
(\textbf{PRP}) (\textbf{UMR}) and (\textbf{SNR}) can be oriented with PLPOs.
However, in contrast to the LPO, it can be shown that the PLPO only induces primitive recursive upper bounds
on derivation lengths of compatible rewrite systems (Corollary \ref{c:bound}).
This yields an alternative proof of the fact that the class of primitive
recursive functions is closed under (\textbf{PRP}) (\textbf{UMR}) and
(\textbf{SNR}) (Corollary \ref{c:main}).

\subsection{Related work}
\label{subs:related}

The recursion-theoretic characterisation of primitive recursive
functions given in \cite{Simmons88} is based on
a restrictive (higher order primitive) recursion that is commonly known as {\em
predicative recursion} \cite{bel_cook} or {\em ramified recursion} \cite{leivant}.
Predicative recursion is a syntactic restriction of the standard
(primitive) recursion based on a separation of argument positions into
two kinds, where the number of recursive calls is measured
only by an  argument occurring left to semicolon whereas results of recursion are allowed to
be substituted only for arguments occurring right:
\begin{equation}
\tag{\textbf{Predicative Recursion}}
\label{e:sr}
\begin{array}{rcl}
f(0, \vec y; \vec z) &=& g(\vec y; \vec z) \\
f(x+1, \vec y; \vec z) &=& h(x, \vec y; \vec z, f(x, \vec y; \vec z))
\end{array}
\end{equation}

The {\em polynomial path order} (POP* for short)
\cite{AM08} is defined to be compatible with
predicative recursion:
$f(s(x), \vec y; \vec z) >_{\pop}
 h(x, \vec y; \vec z, f(x, \vec y; \vec z))$.
It is worth noting that predicative recursion does not make sense with the usual
composition since argument positions can be shifted from left to right.
For example, given a function $f(x; y)$, the function $f' (; x, y)$ 
such that $f(x; y) = f' (; x, y)$ could be defined as
$f' (; x, y) = f ( I^2_1 (; x, y); I^2_2 (; x,  y))$,
and thus the intended argument separation would break.
To maintain the constraint on the argument separation, composition is limited to a restrictive form called {\em predicative composition}:
\begin{equation}
\tag{\textbf{Predicative Composition}}
\label{e:pcom}
f(\vec x; \vec y) = h(\vec g (\vec x;); \vec g' (\vec x; \vec y))
\end{equation}
An auxiliary suborder $\sqsupset_{\pop}$ of the POP* $>_{\pop}$ is defined so that
$f(\vec x; y) \sqsupset_{\pop} g_j (\vec x; )$
holds for each $g_j \in \vec g$.
The POP* induces the polynomial (innermost) {\em runtime complexity} of
compatible rewrite systems.
Namely, for any rewrite system $\mathcal R$ compatible with an instance of POP*,
there exists a polynomial such that the length of any innermost
$\mathcal R$-rewriting sequence starting with a term whose arguments are
normalised can be bounded by the polynomial in the size of the starting term.
Moreover, predicative recursion can be extended to (nested) multiple
recursion, e.g.,
\begin{equation}
\label{e:snr}
\begin{array}{rcl}
f(0, \vec y; \vec z) &=& g(\vec y; \vec z) \\
f(x+1, \vec y; \vec z) &=& 
h \left( x, \vec y; \vec z, 
  f \left( x, \vec p (x, \vec y;); 
    \vec h' (x, \vec y; \vec z, f(x, \vec p' (x, \vec y;); \vec z))
    \right)
\right).
\end{array}
\end{equation}

Essentially, the {\em exponential path order} (EPO* for short)
\cite{AEM11} is defined to be compatible with predicative multiple recursion.
The EPO* induces the exponential innermost runtime complexity of
compatible rewrite systems.
Not surprisingly, the EPO* is too weak to orient the
general form of primitive recursion.
In \cite{Simmons88} the meaning of predicative recursion is relaxed 
(though \cite{Simmons88} is an earlier work than \cite{bel_cook,leivant})
in such a way that recursive calls are still restrictive for (nested) multiple
recursion as in the equation (\ref{e:snr}) but allowed even on
arguments occurring right to semicolon for
the standard primitive recursion, i.e.,
\begin{equation}
\label{e:prec}
\begin{array}{rcl}
f(; 0, \vec y) &=& g(; \vec y) \\
f(; x+1, \vec y) &=& 
h(; x, \vec y, f(; x, \vec y)).
\end{array}
\end{equation}
Intuitively, every primitive recursive function can
be used as an initial function in the underlying function algebra.

\subsection{Outline}
\label{subs:outline}

In Section \ref{s:PLPO} we start with defining an auxiliary suborder
$\sqsupset_{\rrpo}$ of the PLPO $>_{\rrpo}$ (Definition \ref{d:ssup}), which is exactly the same as
$\sqsupset_{\pop}$.
The definition of $>_{\rrpo}$ contains three important cases:
(i) Case \ref{d:rrpo:3} of Definition \ref{d:rrpo} makes it possible to
orient the equation of (\ref{e:pcom});
(ii) Case \ref{d:rrpo:4} makes the orientation
$f(; s(x), \vec y) >_{\rrpo} f(; x, \vec y)$ possible.
This together with Case \ref{d:rrpo:3} makes it possible to
orient the equation (\ref{e:prec}) of primitive recursion;
(iii) Case \ref{d:rrpo:5} makes the orientation
$f(s(x), \vec y; \vec z) >_{\rrpo}
 f \left( x, \vec p (x, \vec y;); 
    \vec h' (x, \vec y; \vec z, f(x, \vec p' (x, \vec y;); \vec z))
   \right)
$
possible as well as
$f(s(x), \vec y; \vec z) >_{\rrpo}
 f \left( x, \vec p' (x, \vec y;); \vec z) \right)
$.
This together with Case \ref{d:rrpo:3} makes it possible to
orient the restrictive form (\ref{e:snr}) of nested recursion.
Without Case \ref{d:rrpo:4}, the PLPO only induces elementary recursive
derivation lengths.

In Section \ref{s:int} we present a primitive recursive interpretation for
the PLPO.
This yields that the maximal derivation length of a
rewrite system compatible with a PLPO
is bounded by a primitive recursive
function in the size of a starting term.

In Section \ref{s:application} we show that the complexity result about
the PLPO obtained in Section \ref{s:int} can be used to show
that the class of primitive recursive functions is closed under (\textbf{PRP}), (\textbf{UMR}) and (\textbf{SNR}).

In Section \ref{s:comparison} we compare the PLPO with related
termination orders to make the contribution of this work clearer.

\section{Predicative lexicographic path orders}
\label{s:PLPO}

Let $\mathcal V$ denote a countably infinite set of variables. 
A {\em signature} $\mathcal F$ is a finite set of function symbols.
The number of argument positions of a function symbol $f \in \mathcal F$
is denoted as $\ar (f)$.
We write $\mathcal{T(V,F)}$ to denote the set of terms over 
$\mathcal V$ and $\mathcal F$ whereas write $\mathcal{T(F)}$ to denote 
the set of closed terms over $\mathcal F$, or the set of {\em ground}
terms in other words.
The signature $\mathcal F$ can be partitioned into the set
$\mathcal C$ of {\em constructors} and the set 
$\mathcal D$ of {\em defined} symbols.
We suppose that $\mathcal C$ contains at least one constant.
The set $\mathcal D$ of defined symbols includes a (possibly empty)
specific subset $\mathcal D_{\lex}$, where a term will be compared
lexicographically if
its root symbol belongs to $\mathcal D_{\lex}$.
A {\em precedence} $\geqslant_{\mathcal F}$ on the signature 
$\mathcal F$ is a quasi-order on $\mathcal F$ whose strict part $>_{\mathcal F}$  
is well-founded.
We write $f \approx_{\mathcal F} g$ if 
$f \geqslant_{\mathcal F} g$ and $g \geqslant_{\mathcal F} f$
hold.

In accordance with the argument separation for predicative recursion, we assume that the argument positions of every function symbol
are separated into two kinds.
As in the schema (\ref{e:sr}), the separation is denoted by semicolon as
$f(t_1, \dots, t_k; t_{k+1}, \dots, t_{k+l})$,
where $t_1, \dots, t_k$ are called {\em normal} arguments whereas
$t_{k+1}, \dots, t_{k+l}$ are called {\em safe} ones.
The equivalence $\approx_{\mathcal F}$ is extended to the term
equivalence $\approx$.
We write 
$f(s_1, \dots, s_k; s_{k+1}, \dots, s_{k+l}) \approx 
 g(t_1, \dots, t_k; s_{k+1}, \dots, t_{k+l})$
if $f \approx_{\mathcal F} g$ and 
$s_j \approx t_j$ for all $j \in \{ 1, \dots, k+l \}$.

\begin{definition}
\label{d:ssup}
An auxiliary relation
$s = f(s_1, \dots, s_k; s_{k+1}, \dots, s_{k+l}) \ssup t$
holds if one of the following cases holds, where
$s \ssupeq t$ denotes $s \ssup t$ or $s \approx t$.
\begin{enumerate}
\item $f \in \mathcal C$ and $s_i \ssupeq t$ for some
      $i \in \{ 1, \dots, k+l \}$.
\label{d:ssup:1}
\item $f \in \mathcal D$ and $s_i \ssupeq t$ for some
      $i \in \{ 1, \dots, k \}$.
\label{d:ssup:2}
\item $f \in \mathcal D$ and 
      $t = g(t_1, \dots, t_m; t_{m+1}, \dots, t_{m+n})$ for some $g$
      such that $f >_{\mathcal F} g$, and
      $s \ssup t_j$ for all $j \in \{ 1, \dots, m+n \}$.
\label{d:ssup:3}
\end{enumerate}
\end{definition}

Now we define the {\em predicative lexicographic path order} (PLPO for short)
denoted as $>_{\rrpo}$.
We write $s \geqslant_{\rrpo} t$ if
$s >_{\rrpo} t$ or $s \approx t$ holds, like the relation $\ssupeq$, write
$(s_1, \dots, s_k) \geqslant_{\rrpo} (t_1, \dots, t_k)$ if
$s_j \geqslant_{\rrpo} t_j$ for all $j \in \{ 1, \dots, k \}$, and we write
$(s_1, \dots, s_k) >_{\rrpo} (t_1, \dots, t_k)$ if
$(s_1, \dots, s_k) \geqslant_{\rrpo} (t_1, \dots, t_k)$ and additionally
$s_i >_{\rrpo} t_i$ holds for some $i \in \{ 1, \dots, k \}$.

\begin{definition}
\label{d:rrpo}
The relation
$s = f(s_1, \dots, s_k; s_{k+1}, \dots, s_{k+l}) >_{\rrpo} t$
holds if one of the following cases holds.
\begin{enumerate}
\item $s \ssup t$.
\label{d:rrpo:1}
\item $s_i \geqslant_{\rrpo} t$ for some $i \in \{ 1, \dots, k+l \}$.
\label{d:rrpo:2}
\item $f \in \mathcal D$ and 
      $t = g(t_1, \dots, t_m; t_{m+1}, \dots, t_{m+n})$ for some
      $g$ such that $f >_{\mathcal F} g$, 
      $s \ssup t_j$ for all $j \in \{ 1, \dots, m \}$, and
      $s >_{\rrpo} t_j$ for all $j \in \{ m+1, \dots, m+n \}$.
\label{d:rrpo:3}
\item $f \in \mathcal D \setminus \mathcal D_{\lex}$ and
      $t = g(t_1, \dots, t_k; t_{k+1}, \dots, t_{k+l})$ for some
      $g$ such that $f \approx_{\mathcal F} g$, 
  \begin{itemize}
  \item $(s_{1}, \dots, s_{k}) \geqslant_{\rrpo}
         (t_{1}, \dots, t_{k})$,
        and   
  \item $(s_{k+1}, \dots, s_{k+l}) >_{\rrpo}
         (t_{k+1}, \dots, t_{k+l})$.
  \end{itemize}
\label{d:rrpo:4}
\item $f \in \mathcal D_{\lex}$ and
      $t = g(t_1, \dots, t_m; t_{m+1}, \dots, t_{m+n})$ for some
      $g$ such that $f \approx_{\mathcal F} g$, and there exists 
      $i_0 \in \{ 1, \dots, \min (k, m) \}$ such that
  \begin{itemize}
  \item $s_j \approx t_j$ for all $j \in \{ 1, \dots, i_0 -1 \}$,
  \item $s_{i_0} >_{\rrpo} t_{i_0}$,
  \item $s \ssup t_j$ for all $j \in \{ i_0 +1, \dots, m \}$, and
  \item $s >_{\rrpo} t_j$ for all $j \in \{ m+1, \dots, m+n \}$.
  \end{itemize}
\label{d:rrpo:5}
\end{enumerate}
\end{definition}

Let $>_{\rrpo}$ be the PLPO induced by a precedence 
$\geqslant_{\mathcal F}$.
Then, by induction according to the definition of $>_{\rrpo}$,
it can be shown that $>_{\rrpo} \subseteq >_{\mathsf{lpo}}$ holds
for the lexicographic path order $>_{\mathsf{lpo}}$ induced by the same
precedence $\geqslant_{\mathcal F}$.
The converse inclusion does not hold in general.

\begin{example}
\label{e:1}
      $\mathcal R_{\mathsf{PR}} = 
       \left\{
       \begin{array}{rcl}
          \mf (; 0, y) & \rightarrow & \mg (;y), \\
          \mf (; \ms (;x), y) & \rightarrow & \mh(; x, y, \mf(;x, y))
       \end{array} 
       \right\}$.

The sets $\mathcal C$ and $\mathcal D$ are defined by 
$\mathcal C = \{ 0, \ms \}$ and $\mathcal D = \{ \mg, \mh, \mf \}$.
Let $\mathcal D_{\lex} = \emptyset$.
Define a precedence $\geqslant_{\mathcal F}$ by
$\mf \approx_{\mathcal F} \mf$ and $\mf >_{\mathcal F} \mg, \mh$.
Define an argument separation as indicated in the rules.
Then $\mathcal R_{\mathsf{PR}}$ can be oriented with the PLPO 
$>_{\rrpo}$ induced by $\geqslant_{\mathcal F}$
as follows.
For the first rule 
$\mf(;0, y) >_{\rrpo} y$ and hence
$\mf(; 0, y) >_{\rrpo} \mg(; y)$ by Case \ref{d:rrpo:3} in Definition
      \ref{d:rrpo}.
Consider the second rule.
Since $(\ms(;x), y) >_{\rrpo} (x, y)$,
$\mf(; \ms(;x), y) >_{\rrpo} \mf(; x, y)$ holds as an instance of Case
      \ref{d:rrpo:4}.
An application of Case \ref{d:rrpo:3} allows us to conclude
$\mf(; \ms(; x), y) >_{\rrpo} \mh(; x, y, \mf(; x, y))$.
\end{example}

\begin{example}
\label{e:2}
$\mathcal R_{\mathsf{PRP}} =
       \left\{
       \begin{array}{rcl}
          \mf(0; y) & \rightarrow & \mg(; y), \\ 
          \mf(\ms (; x); y) & \rightarrow &
          \mh(x; y, \mf(x; \mP (x; y)))
       \end{array}
       \right\}$.

The sets $\mathcal C$ and $\mathcal D$ are defined as in the previous example.
Define the set $\mathcal D_{\lex}$ by
$\mathcal D_{\lex} = \{ \mf \}$.
Define a precedence $\geqslant_{\mathcal F}$ by
$\mf \approx_{\mathcal F} \mf$ and
$\mf >_{\mathcal F} \mQ$ for all 
$\mQ \in \{ \mg, \mQ, \mh \}$.
Define an argument separation as indicated.
Then $\mathcal R_{\mathsf{PRP}}$ can be oriented with the induced PLPO
$>_{\rrpo}$.
We only consider the most interesting case.
Namely we orient the second rule.
Since $\ms(; x) \ssup x$, 
$\mf( \ms(; x); y) \ssup x$ holds by the definition of $\ssup$.
This together with Case \ref{d:rrpo:3} yields
$\mf( \ms(; x); y) >_{\rrpo} \mP (x; y)$.
Hence an application of Case \ref{d:rrpo:5} yields
$\mf( \ms(; x); y) >_{\rrpo} \mf(x; \mP (x; y))$.
Another application of Case \ref{d:rrpo:3} allows us to conclude
$\mf(\ms(; x); y) >_{\rrpo} \mh(x; y, \mf(x; \mP (x; y)))$.
\end{example}

\begin{example}
\label{e:3}
$\mathcal R_{\mathsf{UMR}} =
       \left\{
       \begin{array}{rcl}
          \mf(0, y;) & \rightarrow & \mg_0 (y;), \\
          \mf( \ms(; x), 0; ) & \rightarrow & \mg_1 (x; \mf(x, \mQ (x;); )), \\
          \mf( \ms(; x), \ms(; y);) & \rightarrow & 
          \mh(x, y; \mf(x, \mP (x, y;);), \mf( \ms(; x), y;)) 
       \end{array}
       \right\}$.

The sets $\mathcal C$ and $\mathcal D$ are defined
 as in the former two examples and the set $\mathcal D_{\lex}$ is
 defined as in the previous example. 
Define a precedence $\geqslant_{\mathcal F}$ by
$\mf \approx_{\mathcal F} \mf$ and 
$\mf >_{\mathcal F} \mg$ for all 
$\mg \in \{ \mg_0, \mg_1, \mP, \mQ, \mh \}$.
Define an argument separation as indicated.
Then $\mathcal R_{\mathsf{UMR}}$ can be oriented with the induced PLPO
$>_{\rrpo}$.
Let us consider the most interesting case. 
Namely we orient the third rule.
Since $\mf >_{\mathcal F} \mP$ and $\ms(; u) \ssup u$ for each $u \in \{ x, y\}$,
$\mf( \ms(; x), \ms(; y); ) \ssup \mP (x, y; )$ holds by the definition of $\ssup$.
Hence, since $\ms (; x) >_{\rrpo} x$, an application of Case \ref{d:rrpo:5} in
      Definition \ref{d:rrpo} yields
$\mf( \ms(; x), \ms(; x); ) >_{\rrpo} \mf(x, \mP (x, y; ); )$.
Similarly another application of Case \ref{d:rrpo:5} yields
$\mf( \ms(; x), \ms(; y); ) >_{\rrpo} \mf( \ms(; x), y; )$.
Clearly
$\mf( \ms(; x), \ms(; y); ) \ssup u$ for each $u \in \{ x, y \}$.
Hence an application of Case \ref{d:rrpo:3} allows us to conclude
$\mf( \ms(; x), \ms(; y); )$ $>_{\rrpo}$
$\mh(x, y; \mf(x, \mP (x, y;); )$, $\mf( \ms(; x), y; ))$.
\end{example}

\begin{example}
\label{e:4}
$\mathcal R_{\mathsf{SNR}} =
       \left\{
       \begin{array}{rcl}
          \mf(0; y) & \rightarrow & \mg(; y), \\
          \mf( \ms(; x); y) & \rightarrow &
          \mh(x; y, \mf(x; \mP (x; y, \mf(x; y))))
       \end{array}
       \right\}$.

The sets $\mathcal C$ and $\mathcal D$ are defined as in the former
 three examples and the set $\mathcal D_{\lex}$ is defined as in the
former two examples.
Define a precedence $\geqslant_{\mathcal F}$ as in the previous example.
Define an argument separation as indicated.
Then $\mathcal R_{\mathsf{SNR}}$ can be oriented with the induced PLPO
$>_{\rrpo}$.
We only orient the second rule.
As we observed in the previous example,
$\mf( \ms(; x); y) >_{\rrpo} \mf(x; y)$ holds by Case \ref{d:rrpo:5}.
Hence
$\mf( \ms(; x); y) >_{\rrpo} \mP (x; y, \mf(x; y))$ holds
by Case \ref{d:rrpo:3}.
This together with Case \ref{d:rrpo:5} yields
$\mf( \ms(; x); y) >_{\rrpo} \mf(x; \mP (x; y, \mf(x; y)))$.
Thus another application of Case \ref{d:rrpo:3} allows us to conclude
$\mf( \ms(; x); y) >_{\rrpo}
 \mh(x; y, \mf(x; \mP (x; y, \mf(x; y))))$.
\end{example}

Careful readers may observe that the general form of nested recursion,
e.g., the defining equations for the Ackermann function, cannot be oriented
with any PLPO.

\section{Primitive recursive upper bounds for predicative lexicographic
 path orders}
\label{s:int}

In this section we present a primitive recursive interpretation for
the PLPO.
This yields that the maximal derivation length of a
rewrite system compatible with a PLPO
is bounded by a primitive recursive
function in the size of a starting term.

\begin{definition}
Let $\ell$ be a natural such that $2 \leq \ell$.
Then we define a restriction $\ssup^\ell$ of $\ssup$.
The relation
$s = f(s_1, \dots, s_k; s_{k+1}, \dots, s_{k+l}) \ssup^\ell t$
holds if one of the following cases holds, where
$s \ssupeq^\ell t$ denotes $s \ssup^\ell t$ or $s \approx t$.
\begin{enumerate}
\item $f \in \mathcal C$ and $s_i \ssupeq^\ell t$ for some
      $i \in \{ 1, \dots, k+l \}$.
\item $f \in \mathcal D$ and $s_i \ssupeq^\ell t$ for some
      $i \in \{ 1, \dots, k \}$.
\item $f \in \mathcal D$ and 
      $t = g(t_1, \dots, t_m; t_{m+1}, \dots, t_{m+n})$ for some $g$
      such that $f >_{\mathcal F} g$, and
      $s \ssup^{\ell -1} t_j$ for all $j \in \{ 1, \dots, m+n \}$.
\end{enumerate}
\end{definition}

We write $>_{\rrpo}^\ell$ to denote the PLPO induced by $\ssup^{\ell}$.
The {\em size} of a term $t$, which is the number of nodes in the standard tree representation of $t$,
is denoted as $|t|$.
Note that for any two terms $s, t \in \mathcal{T(F,V)}$, if $s \approx t$,
then $|s| = |t|$ holds.
Following \cite[page 214]{cich_weier}, we define a primitive recursive
function $F_m$ ($m \in \mathbb N$).

\begin{definition}
Given a natural $d \geq 2$, the function $F_m$ is defined by 
$F_0 (x) = d^{x+1}$
and
$F_{m+1} (x) = F_m^{d (1+x)} (x)$,
where $F_m^d$ denotes the $d$-fold iteration of $F_m$.
\end{definition}

For basic properties of the function $F_m$, we kindly refer readers to 
\cite[Lemma 5.4, page 216]{cich_weier}.
Note in particular that $F_m$ is strictly increasing and hence
$F_m (x) + y \leq F_m (x+y)$ holds for any $x$ and $y$.

\begin{definition}
Given a natural $k \geq 1$, we inductively define a primitive recursive function 
$F_{m, n}: \mathbb N^k \rightarrow \mathbb N$ by
\begin{eqnarray*}
F_{m, 0} (x_1, \dots, x_{k}) &=& 0, \\
F_{m, n+1} (x_1, \dots, x_{k}) &=& 
  \begin{cases}
  F_m^{F_{m, n} (x_1, \dots, x_{k}) + d(1+ x_{n+1})} 
  \left( \textstyle{\sum}_{j=1}^{n+1} x_j \right)
  & \text{if } n < k, \\
  F_m^{F_{m, n} (x_1, \dots, x_{k}) +d} 
  \left( \textstyle{\sum}_{j=1}^{k} x_j \right)
  & \text{if } k \leq n.
  \end{cases} 
\end{eqnarray*} 
\end{definition} 

\begin{lemma}
\label{l:Fmn}
For any $n \geq 1$, 
$F_{m,n} (x_1, \dots, x_k) \leq 
 F_{m+1}^n \left( \textstyle{\sum}_{j=1}^k x_j \right)$
holds.
\end{lemma}

\begin{proof}
By induction on $n \geq 1$.
In the base case, we reason as
$F_{m, 1} (x_1, \dots, x_k) =
 F_m^{d (1 + x_{1})} 
 \left( x_1 \right) \leq
  F_m^{d \left( 1 + \sum_{j=1}^k x_j \right)} 
  \left( \textstyle{\sum}_{j=1}^k x_j \right) =
  F_{m+1} \left( \sum_{j=1}^k x_j \right)$.
For the first induction step, suppose $n < k$. Then
\begin{eqnarray*}
&&
F_{m, n+1} (x_1, \dots, x_k)
\\ &=&
F_{m}^{F_{m, n} (x_1, \dots, x_k)
       + d(1+ x_{n+1})
      }
\left( \textstyle{\sum}_{j=1}^{n+1} x_j \right)
\\ &\leq&
F_{m}^{F_{m+1}^n \left( \sum_{j=1}^k x_j \right)
       + d(1+ x_{n+1})
      }
\left( \textstyle{\sum}_{j=1}^{n+1} x_j \right)
\quad \text{(by induction hypothesis)}
\\ &\leq&
F_{m}^{d \left( 1 + F_{m+1}^n 
                \left( \sum_{j=1}^k x_j
                \right)
         \right)
      }
\left( F_{m+1}^n \left( \textstyle{\sum}_{j=1}^k x_j \right) \right)
=
F_{m+1}^{n+1} \left( \textstyle{\sum}_{j=1}^k x_j \right).
\end{eqnarray*}
The last inequality holds since 
$d x_{n+1} \leq 
 (d-1) \cdot F_{m+1}^n \left( \textstyle{\sum}_{j=1}^k x_j \right)$ 
holds.
The case that $k \leq n$ can be shown in the same way.
\end{proof}

\begin{definition}
Let $2 \leq \ell$, $\mathcal F$ be a signature and let
 $\geqslant_{\mathcal F}$ be a precedence
on $\mathcal F$.
The {\em rank}
$\rk: \mathcal F \rightarrow \mathbb N$ 
of function symbols is defined to be compatible with
$\geqslant_{\mathcal F}$, i.e.,
$\rk (f) \geq \rk (g) \Leftrightarrow f \geqslant_{\mathcal F} g$.
Let
$K :=  \max 
 \left( \{ 2 \} \cup
        \{ k \in \mathbb N \mid f \in \mathcal F \ \text{and} \  f \  
             \text{has} \ k \right.$ 
$\left. \text{normal} \ \text{argument} \
             \text{positions}
        \}
 \right)$. 
Then a (monotone) primitive recursive interpretation 
$\mathcal I: \mathcal{T(F)} \rightarrow \mathbb N$
is defined by
\[
 \mathcal I (t) =
 d^{F_{\rk (f) + \ell, K+1} (\mathcal I (t_1), \dots, \mathcal I (t_k))
   } 
 \cdot \left( \textstyle{\sum}_{j=1}^{l} \mathcal I (t_{k+j}) +1 \right),
\]
where
$t = f(t_1, \dots, t_k; t_{k+1}, \dots, t_{k+l}) \in \mathcal{T(F)}$.
We write
$\mathcal J_n (t)$
to abbreviate
$F_{\rk (f) + \ell, n} (\mathcal I (t_1), \dots, \mathcal I (t_k))$,
i.e.,
$\mathcal I (t) =
 d^{\mathcal J_{K+1} (t)} \cdot 
 \left( \sum_{j=1}^{l} \mathcal I (t_{k+j}) +1 \right)$
holds.
\end{definition}

Let
$t = f(t_1, \dots, t_k; t_{k+1}, \dots, t_{k+l}) \in \mathcal{T(F)}$
and $f \in \mathcal F$ with $k$ normal argument positions.
Since $k \leq K$ holds by the definition of the constant $K$,
$\mathcal J_{K+1} (t) = 
 F_{\rk (f) + \ell
   }^{ \mathcal J_K (t) +d}
 \left( \sum_{j=1}^{k} \mathcal I (t_j) \right)$
holds.
Furthermore,
$F_{\rk (f) + \ell}^{2d}
 \left( \sum_{j=1}^{k} \mathcal I (t_j) \right) \leq
 \mathcal J_K (t)
$
holds since $2 \leq K$.
For any ground terms $s, t \in \mathcal{T(F)}$, it can be shown by
induction on the size $|t|$ of $t$ that if
$s \approx t$, then
$\mathcal I (s) = \mathcal I (t)$ holds.

\begin{theorem}
\label{t:main}
Let $s, t \in \mathcal{T(F,V)}$ and 
$\sigma: \mathcal V \rightarrow \mathcal{T(F)}$ 
be a ground substitution.
Suppose 
$\max \big( \{ \ar (f) +1 \mid f \in \mathcal F \} \cup 
            \{ \ell \cdot (K+2) +2 \} \cup \{ |t|+1 \}
      \big) \leq d$.
If 
$s >_{\rrpo}^\ell t$, then,
for the interpretation $\mathcal I$ induced by $\ell$ and $d$,
$\mathcal I (s \sigma) > \mathcal I (t \sigma)$
holds.
\end{theorem}

\begin{proof}
Let $s, t \in \mathcal{T(F,V)}$ and 
$\sigma: \mathcal V \rightarrow \mathcal{T(F)}$ 
be a ground substitution.
As in the assumption of the theorem, choose a constant $d$ so that 
$\max \big( \{ \ar (f) +1 \mid f \in \mathcal F \} \cup 
            \{ \ell \cdot (K+2) +2 \} \cup \{ |t|+1 \}
      \big) \leq d$.
Suppose 
$s >_{\rrpo}^\ell t$.
Then we show that $\mathcal I (s \sigma) > \mathcal I (t \sigma)$ holds
by induction according to the definition of 
$s >_{\rrpo}^\ell t$.
Let $s = f(s_1, \dots, s_k; s_{k+1}, \dots, s_{k+l})$.

In the base case, $s_i \approx t$ holds for some 
$i \in \{ 1, \dots, k+l \}$.
In this case, 
$\mathcal I (t \sigma) = \mathcal I (s_i \sigma) <
 \mathcal I (s \sigma)$ holds.

The argument to show the induction step splits into several cases depending
on the final rule resulting in $s >_{\rrpo}^\ell t$.

{\sc Case.} $s \ssup^{\ell} t$:
In the subcase that $f \in \mathcal C$, 
$s_i \ssupeq^{\ell} t$ for some
$i \in \{ 1, \dots, k+l \}$, and hence
$s_i \geqslant_{\rrpo}^\ell t $ holds.
By IH (Induction Hypothesis),
$\mathcal I (t \sigma) \leq \mathcal I (s_i \sigma) <
 \mathcal I (s \sigma)$
holds.
The case that $f \in \mathcal D$ follows
 from the following claim.

\begin{claim}
Suppose that $2 \leq \ell' \leq \ell$ holds.
If 
$f \in \mathcal D$ and
$s \ssup^{\ell'} t$,
then
(for the interpretation $\mathcal I$ induced by $\ell$)
 the following inequality holds.
\begin{equation}
 \mathcal I (t \sigma) \leq 
 F_{\rk (f) + \ell}^{\ell' \cdot (K+2)} 
 \left( \textstyle{\sum}_{j=1}^k \mathcal I (s_j \sigma) \right)
\label{e:claim}
\end{equation} 
\end{claim}

By the assumption of the theorem,
$\ell' \cdot (K+2) \leq \ell \cdot (K+2) \leq d$
holds.
This implies 
$\mathcal I (t \sigma) \leq 
 F_{\rk (f) + \ell}^{d} 
 \left( \sum_{j=1}^k \mathcal I (s_j \sigma) \right) \leq
 \mathcal J_K (s \sigma) <
 \mathcal I (s \sigma)$.

\begin{proof}[of Claim]
By induction according to the definition of $\ssup^{\ell'}$.
Write $\mathcal H_{\ell'} (s)$ to abbreviate
$F_{\rk (f) + \ell}^{\ell' \cdot (K+2)} 
 \left( \sum_{j=1}^k \mathcal I (s_j \sigma) \right)$.

{\sc Case.} $s_i \ssupeq^{\ell'} t$ for some
$i \in \{ 1, \dots, k \}$:
Let us observe that  $s_i \ssupeq^{\ell} t$ also holds.
This implies $s_i \geqslant_{\rrpo}^\ell t $.
Hence
$\mathcal I (t \sigma) \leq \mathcal I (s_i \sigma)$
holds by IH for the theorem,
and thus
$\mathcal I (t \sigma) \leq \mathcal H_{\ell'} (s)$
also holds.

{\sc Case.}
$t = g(t_1, \dots, t_m; t_{m+1}, \dots, t_{m+n})$
for some $g \in \mathcal F$ and 
$t_1, \dots, t_{m+n} \in \mathcal{T(F,V)}$ such that
$f >_{\mathcal F} g$ and 
$s \ssup^{\ell' -1} t_j$ for all $j \in \{ 1, \dots, m+n \}$.
By IH for the claim,
$\mathcal I (t_j \sigma) \leq \mathcal H_{\ell' -1} (s)$
holds for all $j \in \{ 1, \dots, m+n \}$.
Since $m+n = \ar (g) \leq d$ by the assumption of the theorem, 
$\sum_{j=1}^{m+n} \mathcal I (t_j \sigma) +d \leq
 d (\mathcal H_{\ell' -1} (s) + 1) \leq 
 F_{\rk (f) + \ell} ( H_{\ell' -1} (s))$
holds. 
This implies
\begin{eqnarray}
&&
\textstyle{\sum}_{j=1}^{m+n} \mathcal I (t_j \sigma)
+d 
\nonumber \\
& \leq &
F_{\rk (f) + \ell}^{(\ell' -1) \cdot (K+2) +1} 
\left( \textstyle{\sum}_{j=1}^k \mathcal I (s_j \sigma) \right)
=
F_{\rk (f) + \ell}^{\ell' \cdot (K+2) -K -1} 
\left( \textstyle{\sum}_{j=1}^k \mathcal I (s_j \sigma) \right).
\label{e:sih}
\end{eqnarray}
On the other side, since $\rk (g) < \rk (f)$ by the definition of the rank $\rk$,
we can find a natural $p$ such that
$\rk (g) + \ell \leq p < \rk (f) + \ell$.
Hence it holds that
$\mathcal I (t \sigma) =
 d^{\mathcal J_{K+1} (t \sigma)} \cdot 
 \left( \textstyle{\sum}_{j=1}^{n} \mathcal I (t_{m+j} \sigma) +1
 \right)
 \leq 
 d^{\mathcal J_{K+1} (t \sigma) +
    \sum_{j=1}^n \mathcal I (t_{m+j} \sigma)
    + 1
   } 
 \leq 
  F_p \left( \mathcal J_{K+1} (t \sigma) +
             \textstyle{\sum}_{j=1}^n \mathcal I (t_{m+j} \sigma)
      \right)$.
Thus, to conclude the claim, it suffices to show that
$F_p \left( \mathcal J_{K+1} (t \sigma) +
            \textstyle{\sum}_{j=1}^n \mathcal I (t_{m+j} \sigma)
     \right)
 \leq \mathcal H_{\ell'} (s)$
holds.
To show this inequality, we reason as follows.
\begin{eqnarray*}
&&
F_p \left( \mathcal J_{K+1} (t \sigma) +
           \textstyle{\sum}_{j=1}^n \mathcal I (t_{m+j} \sigma)
    \right)
\\ &\leq&
F_p \left( F_{p}^{\mathcal J_{K} (t \sigma) +d} 
     \left( \textstyle{\sum}_{j=1}^m \mathcal I (t_j \sigma)
     \right)
     + \textstyle{\sum}_{j=1}^n \mathcal I (t_{m+j} \sigma)
    \right)
\\ &\leq&
F_{p}^{1+ \mathcal J_{K} (t \sigma) +d}
\left(\textstyle{\sum}_{j=1}^{m+n} \mathcal I (t_j \sigma)
\right)
\qquad (\text{by strict increasingness of } F_p)
\\ &\leq&
F_p^{1+ F_{p+1}^K 
     \left( \sum_{j=1}^m \mathcal I (t_j \sigma)
     \right) +d
    }
\left(\textstyle{\sum}_{j=1}^{m+n} \mathcal I (t_j \sigma)
\right)
\qquad \text{(by Lemma \ref{l:Fmn})}
\\ &\leq&
F_p^{1+ F_{p+1}^K 
     \left(\sum_{j=1}^m \mathcal I (t_j \sigma) +d
     \right)
    }
\left( \textstyle{\sum}_{j=1}^{m+n} \mathcal I (t_j \sigma)
\right)
\\ &\leq&
F_p^{1+ F_{p+1}^K 
     \left( F_{p+1}^{\ell' \cdot (K+2) -K -1} 
            \left( \sum_{j=1}^k \mathcal I (s_j \sigma)
            \right)
     \right) 
    }
\left( F_{p+1}^{\ell' \cdot (K+2) -K-1} 
       \left( \textstyle{\sum}_{j=1}^k \mathcal I (s_j \sigma)
       \right)
\right)
\\ &\leq&
F_p^{1+ F_{p+1}^{\ell' \cdot (K+2) -1} 
     \left( \sum_{j=1}^k \mathcal I (s_j \sigma)
     \right)
    }
\left( F_{p+1}^{\ell' \cdot (K+2) -1} 
       \left( \textstyle{\sum}_{j=1}^k \mathcal I (s_j \sigma)
       \right)
\right)
\\ &\leq&
F_{p+1} 
\left( F_{p+1}^{\ell' \cdot (K+2) -1} 
       \left( \textstyle{\sum}_{j=1}^k \mathcal I (s_j \sigma)
       \right)
\right)
=
\mathcal H_{\ell'} (s).
\end{eqnarray*} 
The fifth inequality follows from the inequality (\ref{e:sih}).
\qed
\end{proof}

{\sc Case.}
$s_i >_{\rrpo}^\ell t$ holds for some 
$i \in \{ 1, \dots, k+l \}$:
In this case
$\mathcal I (t \sigma) < \mathcal I (s_i \sigma)$ 
by IH, and hence
$\mathcal I (t \sigma) < \mathcal I (s \sigma)$ 
holds.

{\sc Case.}
$f \in \mathcal D \setminus \mathcal D_{\lex}$ and
$t = g(t_1, \dots, t_k; t_{k+1}, \dots, t_{k+l})$ for some
$g$ such that $f \approx_{\mathcal F} g$, 
$(s_{1}, \dots, s_{k}) \geqslant_{\rrpo}^\ell
         (t_{1}, \dots, t_{k})$,
        and   
$(s_{k+1}, \dots, s_{k+l}) >_{\rrpo}^\ell
         (t_{k+1}, \dots, t_{k+l})$:
In this case, 
$\rk (f) = \rk (g)$, and by IH,
$\mathcal I (t_j \sigma) \leq \mathcal I (s_j \sigma)$
for all $j \in \{ 1, \dots, k+l \}$
and additionally
$\mathcal I (t_i \sigma) < \mathcal I (s_i \sigma)$
for some $i \in \{ k+1, \dots, k+l \}$.
Hence it is easy check that
$\mathcal I (t \sigma) < \mathcal I (s \sigma)$
holds.

It remains to consider Case \ref{d:rrpo:3} and \ref{d:rrpo:5} of Definition
 \ref{d:rrpo}.

\begin{claim}
In Case \ref{d:rrpo:3} and \ref{d:rrpo:5} of Definition
 \ref{d:rrpo}, the following two inequalities hold.
  \begin{eqnarray}
   \mathcal J_{K} (t \sigma) + d & \leq & \mathcal J_{K}(s \sigma),
  \label{c:1} \\
   \mathcal J_{K+1} (t \sigma)
   & \leq &
   F_{\rk (f) + \ell}^{J_K (s \sigma) +d-1} 
   \left( \textstyle{\sum}_{j=1}^k \mathcal I (s_j \sigma)
   \right)
  \label{c:2}
  \end{eqnarray} 
\end{claim}
\begin{proof}[of Claim]
We show the inequality (\ref{c:1}) by case analysis.

{\sc Case.}
$f \in \mathcal D$ and 
$t = g(t_1, \dots, t_m; t_{m+1}, \dots, t_{m+n})$ for some
$g$ such that $f >_{\mathcal F} g$, 
$s \ssup^\ell t_j$ for all $j \in \{ 1, \dots, m \}$, and
$s >_{\rrpo}^\ell t_j$ for all $j \in \{ m+1, \dots, m+n \}$:
Since $\rk (g) < \rk (f)$, as in the proof of the previous claim, we can
 find a natural $p$ such that 
$\rk (g) + \ell \leq p < \rk (f) + \ell$.
By auxiliary induction on $j \in \{ 1, \dots, m-1 \}$ we show the
 following inequality.
\[
 \mathcal J_{j} (t \sigma) + d (1+ \mathcal I (t_{j+1} \sigma))\leq 
 F_{p+1}^{d+(j-1)} 
 \left( \textstyle{\sum}_{i=1}^k \mathcal I (s_i \sigma)
 \right)
\]
By the inequality (\ref{e:claim}) and the
 assumption
$\ell \cdot (K+1) + 2 \leq d$,
for any $j \in \{ 1, \dots, m \}$,
$\sum_{i=1}^m \mathcal I (t_i \sigma) + 
 d(1 + \mathcal I (t_{j} \sigma)) \leq
 (K+1) F_{p+1}^{d-2} 
 (\sum_{i=1}^k \mathcal I (s_i \sigma)) + d
 \leq
 d \left( 1 + F_{p+1}^{d-2} 
          (\sum_{i=1}^k \mathcal I (s_i \sigma))
   \right)
 \leq
 F_{p+1}^{d - 1} 
 (\sum_{i=1}^k \mathcal I (s_i \sigma))
$.
For the base case we reason as follows.
\begin{eqnarray*}
&&
 \mathcal J_1 (t \sigma) + d(1 + \mathcal I (t_2 \sigma)) \\
&\leq&
 F_p^{d(1 + \mathcal I (t_i \sigma))} 
 \big( \textstyle{\sum}_{i=1}^m \mathcal I (t_i \sigma)
 \big)
 + d(1 + \mathcal I (t_2 \sigma)) \\
&\leq&
 F_p^{d(1 + \mathcal I (t_1 \sigma))} 
 \big( \textstyle{\sum}_{i=1}^m 
       \mathcal I (t_i \sigma) + d(1 + \mathcal I (t_2 \sigma))
 \big) \\
&\leq&
 F_p^{d \left( 1 + F_{p+1}^{d-1} 
               \left( \sum_{i=1}^k \mathcal I (s_i \sigma) 
               \right) 
        \right)
     }
 \left( F_{p+1}^{d-1} 
               \left( \textstyle{\sum}_{i=1}^k \mathcal I (s_i \sigma) 
               \right)
 \right) 
=
 F_{p+1}^{d} \left( \textstyle{\sum}_{i=1}^k \mathcal I (s_i \sigma) 
             \right).
\end{eqnarray*}
For the induction step we reason as follows.
\begin{eqnarray*}
&&
 \mathcal J_{j+1} (t \sigma) + d(1 + \mathcal I (t_{j+2} \sigma)) \\
&\leq&
 F_p^{\mathcal J_{j} (t \sigma) + 
      d(1 + \mathcal I (t_{j+1} \sigma))
     } 
 \big( \textstyle{\sum}_{i=1}^m \mathcal I (t_i \sigma)
 \big)
 + d(1 + \mathcal I (t_{j+2} \sigma)) \\
&\leq&
 F_p^{\mathcal J_{j} (t \sigma) + 
      d(1 + \mathcal I (t_{j+1} \sigma))
     } 
 \big( \textstyle{\sum}_{i=1}^m 
       \mathcal I (t_i \sigma) + d(1 + \mathcal I (t_{j+2} \sigma))
 \big) \\
&\leq&
 F_p^{F_{p+1}^{d+(j-1)} 
      \left(\sum_{i=1}^k \mathcal I (s_i \sigma)
      \right)
     }
 \left( F_{p+1}^{d-1} 
               \left( \textstyle{\sum}_{i=1}^k \mathcal I (s_i \sigma) 
               \right)
 \right) 
\ (\text{by induction hypothesis}) \\
&\leq&
 F_{p}^{d \left( 1+
                   F_{p+1}^{d+(j-1)}
                   \left( \sum_{i=1}^k 
                          \mathcal I (s_i \sigma)
                   \right)
            \right)
         } 
 \left( F_{p+1}^{d+(j-1)}
        \left( \textstyle{\sum}_{i=1}^k \mathcal I (s_i \sigma)
        \right)
 \right) \\
&=& F_{p+1}^{d+j}
    \left( \textstyle{\sum}_{i=1}^k \mathcal I (s_i \sigma)
    \right).
\end{eqnarray*}
By similar induction on $j \in \{ m, \dots, K \}$
one can show that
$\mathcal J_{j} (t \sigma) + d \leq
 F_{p+1}^{d + (j-1)} (\sum_{i=1}^k \mathcal I (s_i \sigma))$
holds, where the inequality
$\mathcal J_{m-1} (t \sigma) + d (1+ \mathcal I (t_{m} \sigma))\leq 
 F_{p+1}^{d + (m-2)} (\sum_{i=1}^k \mathcal I (s_i \sigma))$
takes place in the base case.
Thus in particular,
$\mathcal J_{K} (t \sigma) + d \leq
 F_{p+1}^{d + (K-1)} (\sum_{i=1}^k \mathcal I (s_i \sigma)) \leq
 F_{p+1}^{2d} (\sum_{i=1}^k \mathcal I (s_i \sigma)) \leq
 \mathcal J_K (s \sigma)
$
holds.

{\sc Case.}
$f \in \mathcal D_{\lex}$ and
$t = g(t_1, \dots, t_m; t_{m+1}, \dots, t_{m+n})$
for some $g \in \mathcal F$ such that
$f \approx_{\mathcal F} g$:
In this case there exists $i_{0} \leq \min \{ k, m\}$ such that
$s_j \approx t_j$ for all $j \in \{ 1, \dots, i_0 -1 \}$,
$s_{i_0} >_{\rrpo}^{\ell} t_{i_0}$,
$s \ssup^{\ell} t_j$ for all $j \in \{ i_0 +1, \dots, m \}$, and
$s >_{\rrpo}^{\ell} t_j$ for all $j \in \{ m+1, \dots, m+n \}$.
Write $p$ to denote $\rk (f) + \ell$.
Then $p = \rk (g) + \ell$ since $\rk (f) = \rk (g)$.
By induction on $j \in \{ 0, \dots, i_0 -1 \}$
it can be shown that
$\mathcal J_{j} (t \sigma) = \mathcal J_{j} (s \sigma)$
holds.
Since
$\mathcal I (t_{i_0} \sigma) < \mathcal I (s_{i_0} \sigma)$
holds by IH for the theorem,
$\mathcal J_{i_0 -1} (t \sigma) = \mathcal J_{i_0 -1} (s \sigma)$
implies
\begin{equation}
\mathcal J_{i_0 -1} (t \sigma) + d (1+ \mathcal I (t_{i_0} \sigma))
 \leq
\mathcal J_{i_0 -1} (s \sigma) + d \cdot \mathcal I (s_{i_0} \sigma).
\label{e:subclaim:base}
\end{equation} 
By auxiliary induction on $j \in \{ i_0, \dots, m-1 \}$ we show that
$\mathcal J_{j} (t \sigma) + d (1 + \mathcal I (t_{j+1} \sigma))
 \leq \mathcal J_{j} (s \sigma)
$
holds.
From the inequalities (\ref{e:claim}),
$\mathcal I (t_{i_0} \sigma) < \mathcal I (s_{i_0} \sigma)$
and $\ell \cdot (K+2) +1 \leq d$,
as in the previous case,
for any $j \in \{ 1, \dots, m \}$,
one can show that
$\sum_{i=1}^{m} \mathcal I (t_i \sigma)
 + d (1+ \mathcal I (t_{j} \sigma))
 \leq
 F_p^{d} 
 \left( \sum_{i=1}^{k} \mathcal I (s_i \sigma) 
 \right)
$.
Assuming this inequality, for the base case we reason follows.
\begin{eqnarray*}
&&
\mathcal J_{i_0} (t \sigma) + d (1+ \mathcal I (t_{i_0 +1} \sigma)) 
\\ &\leq&
 F_p^{\mathcal J_{i_0 -1} (t \sigma) + d (1+ \mathcal I (t_{i_0} \sigma))
     }
 \big( \textstyle{\sum}_{i=1}^{m} \mathcal I (t_{i} \sigma)
 \big)
 + d (1+ \mathcal I (t_{i_0 +1} \sigma))
\\ &\leq&
 F_p^{\mathcal J_{i_0 -1} (s \sigma) + d \cdot \mathcal I (s_{i_0} \sigma)
     }
 \big( \textstyle{\sum}_{i=1}^{m} \mathcal I (t_i \sigma)
 \big)
 + d (1+ \mathcal I (t_{i_0 +1} \sigma))
\ \text{(by the inequality (\ref{e:subclaim:base}))}
\\ &\leq&
 F_p^{\mathcal J_{i_0 -1} (s \sigma) + d \cdot \mathcal I (s_{i_0} \sigma)
     }
 \big( \textstyle{\sum}_{i=1}^{m} \mathcal I (t_i \sigma)
        + d (1+ \mathcal I (t_{i_0 +1} \sigma))
 \big)
\\ &\leq&
 F_p^{\mathcal J_{i_0 -1} (s \sigma) + d \cdot \mathcal I (s_{i_0} \sigma)
     }
 \left( F_p^{d} 
        \left( \textstyle{\sum}_{i=1}^{k} \mathcal I (s_i \sigma)
        \right)
 \right)
\\ &=&
 F_p^{\mathcal J_{i_0 -1} (s \sigma) + d (1+ \mathcal I (s_{i_0} \sigma))
     }
 \left( \textstyle{\sum}_{i=1}^{k} \mathcal I (s_i \sigma)
 \right)
=
\mathcal J_{i_0} (s \sigma).
\end{eqnarray*}
For the induction step we reason as follows.
\begin{eqnarray*}
&&
\mathcal J_{j+1} (t \sigma) + d (1+ \mathcal I (t_{j+2} \sigma)) 
\\ &\leq&
 F_p^{\mathcal J_{j} (t \sigma) + d (1+ \mathcal I (t_{j+2} \sigma))
     }
 \big( \textstyle{\sum}_{i=1}^{m} \mathcal I (t_{i} \sigma)
 \big)
 + d (1+ \mathcal I (t_{j+2} \sigma))
\\ &\leq&
 F_p^{\mathcal J_{j} (s \sigma)
     }
 \big( \textstyle{\sum}_{i=1}^{m} \mathcal I (t_i \sigma)
 \big)
 + d (1+ \mathcal I (t_{j+2} \sigma))
\qquad \text{(by induction hypothesis)}
\\ &\leq&
 F_p^{\mathcal J_{j} (s \sigma)
     }
 \big( \textstyle{\sum}_{i=1}^{m} \mathcal I (t_i \sigma)
        + d (1+ \mathcal I (t_{i_0 +1} \sigma))
 \big)
\\ &\leq&
 F_p^{\mathcal J_{j} (s \sigma)
     }
 \left( F_p^{d} 
        \left( \textstyle{\sum}_{i=1}^{k} \mathcal I (s_i \sigma)
        \right)
 \right)
\\ &=&
 F_p^{\mathcal J_{j} (s \sigma) + d
     }
 \left( \textstyle{\sum}_{i=1}^{k} \mathcal I (s_i \sigma)
 \right)
\leq
\mathcal J_{i_0} (s \sigma).
\end{eqnarray*}
A similar induction on
$j \in \{ m, \dots, K \}$
allows one to deduce
$\mathcal J_j (t \sigma) +d \leq \mathcal J_j (s \sigma)$.
Thus, in particular,
$\mathcal J_{K} (t \sigma) + d \leq \mathcal J_{K}(s \sigma)$
holds.

Let 
$t= g(t_1, \dots, t_m; t_{m+1}, \dots, t_{m+n})$
and $p = \rk (f) + \ell$.
Then the inequality (\ref{c:2}) is shown employing the inequality (\ref{c:1}) as follows.
\begin{eqnarray*}
&&
\mathcal J_{K+1} (t \sigma) \\
&=&
F_p^{\mathcal J_{K} (t \sigma) + d} 
\left( \textstyle{\sum}_{j=1}^m \mathcal I (t_j \sigma)
\right)
\\ 
&\leq&
F_p^{\mathcal J_{K} (s \sigma)} 
\left( \textstyle{\sum}_{j=1}^m \mathcal I (t_j \sigma)
\right)
\qquad (\text{by the inequality (\ref{c:1})}) \\
&\leq&
F_p^{\mathcal J_{K} (s \sigma)} 
\left( F_p^{d-1} 
       \left( \textstyle{\sum}_{j=1}^k \mathcal I (s_j \sigma)
       \right)
\right)
=
F_p^{\mathcal J_K (s \sigma) +d-1} 
\left( \textstyle{\sum}_{j=1}^k \mathcal I (s_j \sigma)
\right).
\end{eqnarray*} 
Note that the last inequality follows from the inequality (\ref{e:claim}).
\qed
\end{proof}

Let us turn back to the proof of the theorem.
In the remaining two cases, instead of showing 
$\mathcal I (t \sigma) < \mathcal I (s \sigma)$ 
directly, 
by subsidiary induction on the size $|t|$ of the term $t$,
we show the following inequality holds. 
\begin{equation} 
\mathcal I (t \sigma) \leq
d^{|t| \cdot 
   \left( 1 + F_{\rk (f) + \ell}^{\mathcal J_K (s \sigma) +d-1} 
          \left( \sum_{j=1}^k \mathcal I (s_j \sigma)
          \right)
   \right)
  }
\cdot
\left( \textstyle{\sum}_{j=1}^l \mathcal I (s_{k+j} \sigma) +1
\right)
\label{e:main}
\end{equation}

Write $p$ to denote $\rk (f) + \ell$.
Since $|t| < d$ holds by the assumption of the theorem and 
$d(1+x) \leq F_p (x)$ holds for any $x$,
the inequality (\ref{e:main}) allows us to conclude that
$\mathcal I (t \sigma) < 
 d^{F_p \left( F_{p}^{\mathcal J_K (s \sigma) +d-1} 
               \left( \sum_{j=1}^k \mathcal I (s_j \sigma)
               \right)
        \right)
   }
\cdot
\left( \sum_{j=1}^l \mathcal I (s_{k+j} \sigma) +1
\right)
= \mathcal I (s \sigma)$.
Case \ref{d:rrpo:3} and \ref{d:rrpo:5} of Definition
 \ref{d:rrpo} can be reduced to the
 following case.

{\sc Case.}
$t = g(t_1, \dots, t_m; t_{m+1}, \dots, t_{m+n})$
for some $g \in \mathcal D$ and 
$t_1, \dots, t_{m+n} \in \mathcal{T(F,V)}$ such that
$s >_{\rrpo}^{\ell} t_j$ holds for all $j \in \{ m+1, \dots, m+n \}$:
Since $|t_{m+j}| \leq |t| -1$, subsidiary induction hypothesis
 together with the assumption that
$\max \{ \ar (f) +1 \mid f \in \mathcal F \} \leq d$
 yields
\begin{eqnarray}
&&  \textstyle{\sum}^n_{j=1} \mathcal I (t_{m+j} \sigma) + 1 
\nonumber \\
& \leq &
 d \cdot 
 d^{(|t| -1) \cdot 
    \left( 1 + F_{p}^{\mathcal J_K (s \sigma) +d-1} 
           \left( \sum_{i=1}^k \mathcal I (s_i \sigma)
           \right)
    \right)
   }
 \cdot
 \left( \textstyle{\sum}_{i=1}^l \mathcal I (s_{k+i} \sigma) +1
 \right)
\nonumber \\
&=&
 d^{(|t| -1) \cdot 
    F_{p}^{\mathcal J_K (s \sigma) +d-1} 
           \left( \sum_{i=1}^k \mathcal I (s_i \sigma)
           \right)
    +|t| 
   }
 \cdot
 \left( \textstyle{\sum}_{i=1}^l \mathcal I (s_{k+i} \sigma) +1
 \right).
\label{e:SIH}
\end{eqnarray}
This enables us to reason as follows.
\begin{eqnarray*}
&&
\mathcal I (t \sigma) \\
&=&
d^{J_{K+1} (t \sigma)} 
\cdot 
\left( \textstyle{\sum}_{j=1}^n \mathcal I (t_{m+j} \sigma) +1
\right)
\\ 
&\leq&
d^{F_{p}^{J_K (s \sigma) +d-1} 
   \left( \sum_{j=1}^k \mathcal I (s_j \sigma)
   \right)
  }
\cdot 
\left( \textstyle{\sum}_{j=1}^n \mathcal I (t_{m+j} \sigma) +1
\right)
\ (\text{by the inequality (\ref{c:2})})
\\ &\leq&
d^{F_{p}^{J_K (s \sigma) +d-1} 
   \left( \sum_{j=1}^k \mathcal I (s_j \sigma)
   \right)
  }
\cdot \\
&& \quad \
d^{(|t| -1)
   F_{p}^{\mathcal J_K (s \sigma) +d-1} 
          \left( \sum_{j=1}^k \mathcal I (s_j \sigma)
          \right)
          + |t|
  }
\cdot 
\left( \textstyle{\sum}_{j=1}^l \mathcal I (s_{k+j} \sigma) +1
\right)
\quad (\text{by (\ref{e:SIH})})
\\ &=&
d^{|t| \cdot 
   \left( 1 + F_{p}^{\mathcal J_K (s \sigma) +d-1} 
          \left( \sum_{j=1}^k \mathcal I (s_j \sigma)
          \right)
   \right)
  }
\cdot
\left( \textstyle{\sum}_{j=1}^l \mathcal I (s_{k+j} \sigma) +1
\right).
\end{eqnarray*}
This finalises the proof of the theorem
\qed
\end{proof}

\begin{lemma}
\label{l:context}
Let $s, t \in \mathcal{T(F)}$ be ground terms and 
$C(\Box) \in \mathcal{T(F \cup \{ \Box \})}$
be a (ground) context.
If $\mathcal I (s) > \mathcal I (t)$, then
$\mathcal I (C(s)) > \mathcal I (C(t))$ holds.
\end{lemma}

\begin{proof}
By induction on the size $|C|$ of the given context 
$C \in \mathcal{T(F \cup \{ \Box \})}$.
\end{proof}

For a rewrite system $\mathcal R$, we write 
$\mathcal R \subseteq >_{\rrpo}$
(or $\mathcal R \subseteq >_{\rrpo}^\ell$)
if $l >_{\rrpo} r$ 
(or $l >_{\rrpo}^\ell r$ respectively) holds
for each rewriting rule $l \rightarrow r \in \mathcal R$.

\begin{theorem}
\label{t:embed}
Let $\mathcal R$ be a rewrite system over a signature $\mathcal F$ such
 that $\mathcal R \subseteq >_{\rrpo}^\ell$
for some $\ell \geq 2$
and $s, t \in \mathcal{T(F)}$
be ground terms. 
Suppose
$\max \big( \{ \ar (f) +1 \mid f \in \mathcal F \} \cup
            \{ \ell \cdot (K+2) +2 \} \cup
            \{ |r| +1 \mid \exists l (l \rightarrow r \in \mathcal R) \}
      \big) \leq d$.
If $s \rightarrow_{\mathcal R} t$, then,
for the interpretation induced by $\ell$ and $d$,
$\mathcal I (s) > \mathcal I (t)$
holds.
\end{theorem}

\begin{proof}
By induction according to the rewriting relation 
$\rightarrow_{\mathcal R}$ resulting in
$s \rightarrow_{\mathcal R} t$.
The base case follows from Theorem \ref{t:main} and the induction step
 follows from Lemma \ref{l:context}.
\end{proof}

\begin{corollary}
\label{c:bound}
For any rewrite system $\mathcal R$, if
$\mathcal R \subseteq >_{\rrpo}$ holds for some PLPO $>_{\rrpo}$, 
then the length of any rewriting sequence in $\mathcal R$ starting with
 a ground term is bounded by a primitive recursive function in the
 size of the starting term.
\end{corollary}

\begin{proof}
Given a rewrite system $\mathcal R$, suppose that
$\mathcal R \subseteq >_{\rrpo}$ holds for some PLPO $>_{\rrpo}$.
Let $\ell \geq \max \{ |r| \mid \exists l (l \rightarrow r \in \mathcal R) \}$.
Then it can be seen that even
$\mathcal R \subseteq >_{\rrpo}^{\ell}$ holds.
Choose a constant $d$ so that
$\max \big( \{ \ar (f) +1 \mid f \in \mathcal F \} \cup
            \{ \ell \cdot (K+2) +2 \}
      \big) \leq d$.
Then, for any rewriting rule $l \rightarrow r \in \mathcal R$,
$|r| +1 \leq \ell \cdot (K+2) +2 \leq d$.
Hence by Theorem \ref{t:embed}, for any ground term $t$,
the maximal length of rewriting sequences starting with $t$ is bounded by 
$\mathcal I (t)$ for the interpretation $\mathcal I$ induced by 
$\ell$ and $d$.
It is not difficult to observe that $\mathcal I (t)$ is bounded by 
$F(|t|)$ for a primitive recursive function $F$.
This observation allows us to conclude.
\qed
\end{proof}

\section{Application to non-trivial closure conditions for primitive
 recursive functions}
\label{s:application}

In this section we show that Corollary \ref{c:bound} can be used to show
that the class of primitive recursive functions is closed under primitive
recursion with parameter substitution (\textbf{PRP}), unnested multiple
recursion (\textbf{UMR}) and simple nested recursion (\textbf{SNR}).

\begin{lemma}
\label{l:prec}
For any primitive recursive function $f$ there exists a rewrite system
 $\mathcal R$ defining $f$ such that $\mathcal R \subseteq >_{\rrpo}$
 for some PLPO $>_{\rrpo}$, where the argument positions of every
 function symbol are safe ones only.
\end{lemma}

\begin{proof}
By induction along the primitive recursive definition of $f$.
We always assume that set $\mathcal C$ of constructors consists only of
 a constant $0$ and a unary constructor $\ms$, where $0$ is 
interpreted as the least natural $0$ and $\ms$ as the numerical successor function.

{\sc Case.} $f$ is one of the initial functions:
First consider the subcase that $f$ is the $k$-ary constant function
$(x_1, \dots, x_k) \mapsto 0$.
In this subcase $f$ is defined by a single rule
$\mO^k (x_1, \dots, x_k) \rightarrow 0$.
Defining a precedence $\geqslant_{\mathcal F}$ by
$\mO^k \approx \mO^k$ and
$\mO^k >_{\mathcal F} 0$,
we can see that for the PLPO induced by $\geqslant_{\mathcal F}$,
$\mO^k (; x_1, \dots, x_k) >_{\rrpo} 0$
holds by Case \ref{d:rrpo:3} of Definition \ref{d:rrpo}.
Consider the subcase that $f$ is a $k$-ary projection function
$(x_1, \dots, x_k) \mapsto x_j$
($j \in \{ 1, \dots, k \}$).
In this subcase $f$ is defined by a single rule
$\mI^k_j (x_1, \dots, x_k) \rightarrow x_j$.
An application of Case \ref{d:rrpo:2} of Definition \ref{d:rrpo}
allow us to see that
$\mI^k_j (; x_1, \dots, x_k) >_{\rrpo} x_j$
holds.

{\sc Case.} $f$ is defined from primitive recursive functions 
$\vec g = g_1, \dots, g_l$
 and $h$ by composition as 
$f(x_1, \dots, x_k) = h(g_1 (\vec x), \dots, g_l (\vec x))$:
By IH each of the functions $\vec g$ and $h$ can be defined by a rewrite
 system which can be oriented with a PLPO.
Define a rewrite system $\mathcal R$ by
expanding those rewrite systems obtained from IH with such a new rule as
$\mf (x_1, \dots, x_k) \rightarrow 
 \mh (\mg_1 (\vec x), \dots, \mg_l (\vec x))$.
Then the function $f$ is defined by the rewrite system $\mathcal R$.
Extend the precedences obtained from IH so that
$\mf \approx_{\mathcal F} \mf$ and
$\mf >_{\mathcal F} \mg$
for all $\mg \in \{ \mg_1, \dots, \mg_l, \mh \}$.
Then, for the induced PLPO $>_{\rrpo}$, we have that
$\mf (; \vec x) >_{\rrpo} \mg_j (; \vec x)$
for all $j \in \{ 1, \dots, l \}$
by an application of Case \ref{d:rrpo:3} of Definition  \ref{d:rrpo}.
Hence another application of Case \ref{d:rrpo:3} allows us to conclude
$\mf (; x_1, \dots, x_k) >_{\rrpo}
 \mh (; \mg_1 (; \vec x), \dots, \mg_l (; \vec x))$.

{\sc Case.}
$f$ is defined by primitive recursion: 
In this case it can be seen that the assertion holds as in Example \ref{e:1}.
\qed
\end{proof}

\begin{theorem}
\label{t:snr}
If a function $f$ is defined from primitive recursive functions by an
instance of $\mathbf{(PRP)}$, $(\mathbf{UMR})$ or $(\mathbf{SNR})$,
then there exists a rewrite system
 $\mathcal R$ defining $f$ such that $\mathcal R \subseteq >_{\rrpo}$
 for some PLPO $>_{\rrpo}$.
\end{theorem} 

\begin{proof}
To exemplify the theorem, suppose that $f$ is defined from primitive
 recursive functions $g$, $p$ and $h$ by (\textbf{SNR}) as 
$f(0, y) = g(y)$ and
$f(x+1, y) = h(x, y, f(x, p(x, y, f(x, y))))$.
Let $\mg$, $\mP$ and $\mh$ denote function symbols corresponding
 respectively to the functions $g$, $p$ and $h$.
From Lemma \ref{l:prec} one can find rewrite systems defining $g$, $p$
 and $h$ all of which can be oriented with some PLPO with the argument
 separation $\ms (; x)$, $\mg(; y)$, $\mP (; x, y, z)$ and $\mh (; x, y, z)$.
We expand the signatures for $g$, $p$ and $h$ with new function symbols
$\mP'$, $\mh'$ and $\mf$.
Define a rewrite system $\mathcal R$ by expanding those rewrite systems
 for $g$, $p$ and $h$ with the following four rules.
\begin{equation*}
 \begin{array}{rrclcrrcl}
 (\mathrm i) & \mP'(x; y, z) &\rightarrow& \mP (; x, y, z) &\mbox{\quad}&
 (\mathrm{iii}) & \mf (0; y) &\rightarrow& \mg (; y) \\
 (\mathrm{ii}) & \mh'(x; y, z) &\rightarrow& \mh (; x, y, z) & \mbox{\quad} &
 (\mathrm{iv}) & \mf (\ms (; x); y) &\rightarrow& 
 \mh' (x; y, \mf (x; \mP' (x; y, \mf(x; y))))
 \end{array}
\end{equation*} 
Clearly, $f$ is defined by the rewrite system $\mathcal R$. 
Expand the precedences for $g$, $p$ and $h$ so that
$\mQ' >_{\mathcal F} \mQ$ for each $\mQ \in \{ \mP, \mh \}$,
$\mf \approx_{\mathcal F} \mf$ and
$\mf >_{\mathcal F} \mQ$ for all $\mQ \in \{ \mg, \mP', \mh' \}$.
Finally, we expand the separation of argument positions as
indicated in the new rules (i)--(iv).
Then, as observed in Example \ref{e:4}, the rule (iii) and (iv) can be
 oriented with the induced PLPO $>_{\rrpo}$.
For the rule (i), since $\mP' >_{\mathcal F} \mP$ and
$\mP' (x; y, z) >_{\rrpo} u$ for all $u \in \{ x, y, z \}$,
$\mP' (x; y, z) >_{\rrpo} \mP (; x, y, z)$ holds as an instance of Case
\ref{d:rrpo:3} of Definition \ref{d:rrpo}.
The rule (ii) can be oriented in the same way.
\qed
\end{proof}

\begin{corollary}
\label{c:main}
The class of primitive recursive functions is closed under primitive
 recursion with parameter substitution, unnested multiple recursion and simple nested recursion.
\end{corollary} 

\begin{proof}
Suppose that a function $f: \mathbb N^k \rightarrow \mathbb N$ is
 defined by an instance of either primitive recursion with parameter
 substitution, unnested multiple recursion, or simple nested recursion.
Then by Theorem \ref{t:snr} we can find a rewrite
 system $\mathcal R$ defining $f$ such that 
$\mathcal R \subseteq >_{\rrpo}$ for some PLPO $>_{\rrpo}$.
We can assume that the underlying set $\mathcal C$ of constructors
 consists only of $0$ and $\ms$.
For each natural $m$ let 
$\underline{m} \in \mathcal{T(C)}$ 
denote a ground term defined by
$\underline{0} = 0$ and
$\underline{m+1} = \ms (\underline{m})$.
Then by definition $|\underline{m}| = m$ for any $m$.
Suppose that a function symbol $\mathsf f$ can be interpreted as the function $f$. 
Then by Corollary \ref{c:bound} there exists a primitive recursive function
 $F$ such that the maximal length of rewriting sequences in 
$\mathcal R$ 
starting with $\mathsf f (\underline{m_1}, \dots, \underline{m_k})$
is bounded by $F \left( \sum_{j=1}^k m_j \right)$.
As observed in \cite{cich_weier}, one can find elementary recursive
 functions $G$, $H_0$ and $H_1$ such that
$G(0, \vec x) = H_0 (\vec x)$,
$G(y+1, \vec x) = H_1 (y, \vec x, G(y, \vec x))$,
and $f$ is elementary recursive in the function
$\vec x \mapsto 
 G \left (F \left( \sum_{j=1}^k x_j \right), \vec x \right)$,
where $\vec x = x_1, \dots, x_k$.
Essentially, the stepping function $H_1$ simulates one step of rewriting
 in $\mathcal R$ (assuming a suitable rewriting strategy).
Obviously, the function 
$\vec x \mapsto 
 G \left( F \left( \sum_{j=1}^k x_j \right), \vec x \right)$
is primitive recursive, and thus so is
 $f$.
\qed
\end{proof}

\section{Comparison to related path orders and limitations of
predicative lexicographic path orders}
\label{s:comparison}

The definition of PLPO is strongly motivated by path orders based
on the normal/safe argument separation \cite{AM08,AEM11,AEM12}.
On one hand, due to allowance of multiset comparison in the polynomial path order POP*  \cite{AM08}, PLPO and POP* are
incomparable. 
On the other hand, the PLPO is an extension of the exponential path
order EPO* \cite{AEM11} though the EPO* only induces the exponential
(innermost) runtime complexity.
By induction according to the inductive definition of EPO* 
$>_{\mathsf{epo*}}$, it can be shown that
$>_{\mathsf{epo*}} \subseteq >_{\rrpo}$ holds with the same precedence
and the same argument separation. 
In general, none of (\textbf{PRP}), (\textbf{UMR}) and (\textbf{SNR}) can
be oriented with POP*s or EPO*s.

Readers also might be interested in comparison to the {\em ramified
lexicographic path order} RLPO \cite{Cichon92}, which
covers (\textbf{PRP}) and (\textbf{UMR}) but cannot handle
(\textbf{SNR}).
The contrast to POP*, EPO* and RLPO can be found in
Fig. \ref{fig:comparison}, where ``runtime complexity'' means
innermost one.
Very recently, the {\em generalised ramified lexicographic path order}
GRLPO, which is an extension of the RLPO, has been considered by
A. Weiermann in a manuscript \cite{weier_GRLPO}.
In contrast to the RLPO, the GRLPO can even handle (\textbf{SNR}).
The GRLPO only induces the primitive recursive runtime complexity
like the PLPO, but seems incomparable with the PLPO.

We mention that the PLPO can also handle a slight extension of simple nested recursion with more
than one recursion parameters, known as {\em general} simple nested recursion
\cite[page 221]{cich_weier}.
For example, consider the following two equations of general simple nested recursion:
\begin{eqnarray}
f(x+1, y+1; z) &=& 
h(x, y; z, f(x, p(x, y;); f(x+1, y; z)))
\label{gsnr1} \\
f(x+1, y+1; z) &=& 
h(x, y; z, f(x, p(x, y; z); f(x+1, y; z)))
\label{gsnr2}
\end{eqnarray}
The equation (\ref{gsnr1}) can be oriented with a PLPO, but the equation (\ref{gsnr2}) cannot be oriented with any PLPO
due to an additional occurrence of $z$ in a safe argument position of $p(x, y; z)$.

\begin{figure}[t]
\centering
\begin{tabular}{p{2cm}|c|c|c|c|c}
\hline 
& \ POP* \ & \ EPO* \ & \quad RLPO \quad & PLPO 
& \ LPO \ \\
\hline \hline
runtime complexity &
\ polynomial \ & \ exponential \ & 
\multicolumn{2}{c|}{ \ primitive recursive \ }
& \ multiply recursive \\
\hline
(\textbf{PRP}) & ---  & --- & \ding{`3} & \ding{`3} & \ding{`3} \\
(\textbf{UMR}) & ---  & --- & \ding{`3} & \ding{`3} & \ding{`3} \\
(\textbf{SNR}) & ---  & --- & --- & \ding{`3} & \ding{`3} \\
\hline
\end{tabular} 
\\
(---: not orientable, \ding{`3}: orientable)
\caption{Comparison to related path orders}
\label{fig:comparison}
\end{figure}

\section{Concluding remarks}
\label{s:conclusion}

We introduced a novel termination order, the predicative lexicographic path order PLPO.
As well as the lexicographic path order LPO, any instance of primitive
recursion with parameter substitution (\textbf{PRP}), unnested multiple
recursion (\textbf{UMR}) and simple nested recursion (\textbf{SNR})
 can be oriented with a PLPO.
On the other side, in contrast to the LPO,
the PLPO only induces primitive recursive upper bounds on derivation
lengths of compatible rewrite systems.
Relying on the connection between time complexity of functions and
runtime complexity of rewrite systems, this yields a uniform proof
of the classical fact that the class of primitive recursive functions is
closed under (\textbf{PRP}), (\textbf{UMR}) and (\textbf{SNR}).

It can be seen that the primitive recursive interpretation presented in Section 
\ref{s:int} is not affected by allowing permutation of safe argument
positions in Case \ref{d:rrpo:4} of Definition \ref{d:rrpo}.
Namely, Case \ref{d:rrpo:4} can be replaced with a slightly stronger condition:
{\it 
      $f \in \mathcal D \setminus \mathcal D_{\lex}$,
      $t = g(t_1, \dots, t_k; t_{k+1}, \dots, t_{k+l})$ for some
      $g$ such that $f \approx_{\mathcal F} g$, 
  \begin{itemize}
  \item $(s_{1}, \dots, s_{k}) \geqslant_{\rrpo}
         (t_{1}, \dots, t_{k})$,
        and   
  \item $(s_{k+1}, \dots, s_{k+l}) >_{\rrpo}
         (t_{\pi (k+1)}, \dots, t_{\pi (k+l)})$
        for some permutation 
        $\pi: \{ k+1, \dots, k+l \} \rightarrow
              \{ k+1, \dots, k+l \}$.
  \end{itemize}
\normalfont
Allowance of permutation of normal argument positions is not clear.
One would recall that, as shown by D. Hofbauer in \cite{hof}, the 
{\em multiset path order} only induces primitive recursive upper bounds on derivation lengths
of compatible rewrite systems.
Allowance of multiset comparison is not clear even for safe arguments.
}



\end{document}